\documentclass[letterpaper,11pt]{amsart}

\usepackage{epsfig}
\usepackage{amssymb}
\usepackage{amsfonts}
\usepackage{amsmath}
\usepackage{graphicx}
\usepackage{mathrsfs}
\usepackage{stmaryrd}
\usepackage{amsthm}
\usepackage[all]{xy}
\usepackage[top=1in, bottom=1in, left=1in, right=1in]{geometry}
\usepackage{color}
\usepackage[colorlinks=true]{hyperref}
\hypersetup{urlcolor=blue,citecolor=red,linkcolor=blue}

\newtheorem{theorem}{Theorem}[section]
\newtheorem{lemma}[theorem]{Lemma}
\newtheorem{proposition}{Proposition}[section]

\numberwithin{equation}{section}

\begin{document}
\date{\today}
\title
{Comparison theorems for fractional GJMS operators}

\author{Huihuang Zhou}
\address{School of Mathematical Sciences\\
	University of Science and Technology of China\\
	Hefei, 230026\\ P.R. China\\}
\email{zhouhuihuang@ustc.edu.cn}


\begin{abstract}
In this paper, we mainly focus on the fractional GJMS operators $P_{2\gamma}$ which are defined on the conformal infinity of a Poincar\'{e}-Einstein manifold. We derive two comparison inequalities of the fractional Yamabe constants associated to the fractional GJMS operators. One is between $P_{1}$ and $P_{2\gamma}$ for $\gamma\in (1/2,1)$, and the other is between $P_{2}$ and $P_{2\gamma}$ for $\gamma\in (1,2)$. They both imply the rigidity theorems by characterizing the equalities. Together with the result in \cite{WZ1}, we partially provide some evidence for the monotonicity of the fractional Yamabe constants. 
\end{abstract}

\maketitle

\section{Introduction}

Suppose $X^{n+1}(n\geq 3)$ is a smooth compact manifold with boundary $M^n$. A function $\rho$ is a  defining function of $M$ in $X$ if
\[ 0\leq \rho\in C^{\infty}(\overline{X}),\quad \rho>0 \mathrm{~in~} X,\quad \rho=0 \mathrm{~on~} M,\quad d\rho|_{M}\neq 0.\]
We say $g_+$ is a conformally compact metric on $X$ with conformal infinity $(M,[\hat{g}])$ if there exists a defining function $\rho$ such that $(\overline{X},\bar{g})$ is compact for $\bar{g}=\rho^2 g_+$, and $\bar{g}|_{TM}\in [\hat{g}]$. Then if $g_+$ is Einstein, i.e. 
\[
Ric_{g_+}=-ng_+,
\]
$(X,g_+)$ is called conformally compact Einstein manifold, or Poincar\'{e}-Einstein manifold.
 
Given a Poincar\'{e}-Einstein manifold $(X^{n+1},g_+)$ and a representative $\hat{g}\in [\hat{g}]$ on $M$, there is a uniquely geodesic normal defining function $r$ such that on a boundary collar neighborhood $M\times [0,\epsilon)$, $g_+$ has the normal form
\[
 g_+ =r^{-2}(dr^2 + g_r),
\]
for $g_r$ a one-parameter family of metrics on $M$ with $g_0 =\hat{g}$ and having asymptotic expansion that is even in $r$ at least up to order $n$ according to \cite{CDLS}:
\begin{equation}\label{Geq}
g_r=
\begin{cases}
\hat{g} +r^2 g_2 +\cdots+r^{n-1}g_{n-1}+r^n g_n +O(r^{n+1}),& n \textrm{ is odd},\\
\hat{g} +r^2 g_2 +\cdots+r^{n-2}g_{n-2}+(r^n\log r)h+ r^n g_n +O(r^{n+1}\log r),& n \textrm{ is even}.
\end{cases}
\end{equation}
Here $g_i (1\leq i\leq n-1)$ and $h$ are symmetric 2-tensors determined by $\hat{g}$, and $g_n$ is the global term which can not be locally determined. 


Consider a fractional power $\gamma\in (0,\frac{n}{2})$ and denote $s=\frac{n}{2}+\gamma$, assume $\gamma\notin\mathbb{N}$ and $s(n-s)\notin \mathrm{Spec}(-\Delta_+)$, where $-\Delta_+$ is the Beltrami-Laplacian operator of $g_+$. It is well known (c.f.\cite{MM}, \cite{GZ}) that, given any $f\in C^{\infty}(M)$, then there is a unique solution satisfying the following equation:
\begin{equation}\label{pos}
    -\Delta_+ u -s(n-s) u=0, \quad r^{s-n}u|_M =f,
\end{equation}
Moreover, $u$ takes the form
\[
u=r^{n-s}F+r^s G, \quad F|_M =f,\quad F,G\in C^{\infty}(\overline{X}).
\]
Then the scattering operator $S(s)$ is defined by
\[ S(s)f=G|_{M}.\]
We point out that the spectral condition can be satisfied by imposing some geometric conditions. By \cite{Le} if $(M,[\hat{g}])$ is of nonnegative Yamabe type, then $\lambda_1(-\Delta_+)\geq \frac{n^2}{4}$.
 
Now the fractional GJMS operator of order $2\gamma$ is defined by the normalized scattering operator:
\[
P^{\hat{g}}_{2\gamma}:=d_\gamma S\left(\frac{n}{2}+\gamma\right),\quad \textrm{for } d_\gamma=2^{2\gamma}\frac{\Gamma(\gamma)}{\Gamma(-\gamma)}.
\]
And the \emph{fractional Q-curvature} is 
\[Q^{\hat{g}}_{2\gamma}:=\frac{2}{n-2\gamma}P^{\hat{g}}_{2\gamma}(1).\]

The fractional GJMS operators is a generalization of the conformal power of Laplacian, which is defined in \cite{CC} following the spectrum theory and scattering theory for asymptotical hyperbolic manifolds in \cite{MM} and \cite{JS}. Some special case were also be studied in \cite{CS} by the extension theory. 

We emphasize that $P_{2\gamma}^{\hat{g}}$ is a self-adjoint and conformal covariant pesudo-differential elliptic  with principal symbol $|\xi|_{\hat{g}}^{2\gamma}$. The meromorphic extension of $S(s)$ given in \cite{JS} implies that $P^{\hat{g}}_{2\gamma}$ is a continuous family of operators in the real parameter $\gamma$ as long as $\frac{n^2}{4}-\gamma^2\notin \mathrm{Spec}(-\Delta_+).$
When $\gamma=k\leq \frac{n}{2}$ is a positive integer, it coincides with the classical GJMS operator\cite{GJMS} of order $2k$ by \cite{GZ}, which is  only determined by the boundary metric. 
For example, when $\gamma=1$ then $P^{\hat{g}}_2$ is the conformal Laplacian:
$$
P^{\hat{g}}_2=-\Delta_{\hat{g}}+\frac{n-2}{2}J^{\hat{g}}, \quad \textrm{where } J^{\hat{g}}=\frac{1}{2(n-1)} R^{\hat{g}}. 
$$
and when $\gamma=2$, $P^{\hat{g}}_4$ is the Paneitz operator
\[
P^{\hat{g}}_4=(-\Delta_{\hat{g}})^2 +\delta_{\hat{g}}\left((n-2)J^{\hat{g}}-4A^{\hat{g}}\right)d_{\hat{g}} +\frac{n-4}{2}Q_{4}^{\hat{g}},
\]
where $A^{\hat{g}}$ is the Schouten tensor and
\[
A^{\hat{g}}=\frac{1}{n-2}(Ric_{\hat{g}}-J^{\hat{g}}\hat{g}),\quad Q_{4}^{\hat{g}}=-\Delta_{\hat{g}} J^{\hat{g}}+\frac{n}{2}(J^{\hat{g}})^2 -2|A^{\hat{g}}|^{2}_{\hat{g}}.
\]
Moreover the fractional Yamabe constants $Y_{2\gamma}(M,\hat{g})$ are defined by the following
\[
Y_{2\gamma}(M,\hat{g})=\inf_{\hat{g}\in [\hat{g}]}\frac{\frac{n-2\gamma}{2}\int_{M} Q_{2\gamma}^{\hat{g}}dv_{\hat{g}}}{Vol(M,\hat{g})^{\frac{n-2\gamma}{n}}}=\inf_{f\in C^{\infty}(M), f>0}
\frac{\int_{M} fP^{\hat{g}}_{2\gamma}f \mathrm{dv}_{\hat{g}}}{(\int_{M} f^{\frac{2n}{n-2\gamma}} \mathrm{dv}_{\hat{g}})^{\frac{n-2\gamma}{n}}}.
\]

The classical example is the ball model of the hyperbolic space $\mathbb{H}^{n+1}$, whose conformal infinity is $(\mathbb{S}^n ,[g_{\mathbb{S}}])$, where $g_{\mathbb{S}}$ is the round metric on $\mathbb{S}^n$. In this case, it is easy to know $\textrm{Spec}(-\Delta_{g_{\mathbb{H}}})=[\frac{n^2}{4},+\infty)$. Thus according to \cite{Be1} for any $\gamma\in \left(0,\frac{n}{2}\right)$, $P_{2\gamma}^{g_{\mathbb{S}}}$ can be represented by
\begin{equation}
P_{2\gamma}^{g_{\mathbb{S}}}=\frac{\Gamma(B+\frac{1}{2}+\gamma)}{\Gamma(B+\frac{1}{2}-\gamma)},\quad \textrm{where}\quad B=\sqrt{-\Delta_{g_{\mathbb{S}}}+\frac{(n-1)^2}{4}}.
\end{equation}
Then the fractional Q-curvature is
\[
Q^{g_{\mathbb{S}}}_{2\gamma}=\frac{2}{n-2\gamma}P_{2\gamma}^{g_\mathbb{S}}(1)=\frac{2}{n-2\gamma}\frac{\Gamma(\frac{n}{2}+\gamma)}{\Gamma(\frac{n}{2}-\gamma)},
\]
and the fractional Yamabe constant is 
\[
Y_{2\gamma}(\mathbb{S}^n, [{g_{\mathbb{S}}}]) =2^{\frac{2\gamma}{n}}\pi^{\frac{\gamma(n+1)}{n}} \frac{\Gamma(\frac{n}{2}+\gamma)}{\Gamma(\frac{n}{2}-\gamma)}\left[ \Gamma\left(\frac{n+1}{2}\right)\right]^{-\frac{2\gamma}{n}} =\frac{\Gamma(\frac{n}{2}+\gamma)}{\Gamma(\frac{n}{2}-\gamma)} \left|\mathbb{S}^n\right|^{\frac{2\gamma}{n}}.
\]

We mainly want to study the family properties of these operators by comparing them with the version of the standard hyperbolic model. At the beginning, we recall the comparison theorem between $P^{\hat{g}}_{1}$ and $P^{\hat{g}}_{2}$.
\begin{theorem}[\cite{CLW,WZ1,WW}]\label{mth0}
Suppose $(X^{n+1},g_+)$ $(n\geq 3)$ is a Poincar\'{e}-Einstein manifold, which is $C^{3,\alpha}$ conformally compact with conformal infinity $(\partial X,[\hat{g}])$. Assume $\frac{n^2-1}{4}\notin \mathrm{Spec}(-\Delta_+)$,
then
\begin{equation}\label{ineq0}
\frac{Y_2(M,[\hat{g}])}{Y_2(\mathbb{S}^n,[g_{\mathbb{S}}])}
 \leq \left(\frac{Y_1(M,[\hat{g}])}{Y_1(\mathbb{S}^n,[g_{\mathbb{S}}])}\right)^2,
\end{equation}
and the equality holds if and only if $(X, g_+)$ is isometric to the hyperbolic space $\mathbb{H}^{n+1}$.
\end{theorem}

The inequality (\ref{ineq0}) is conformally invariant and was first proved by \cite{CLW} since $Y_1(M,[\hat{g}])$ is the same as the second Escobar-Yamabe type \cite{Es1,Es2}. Together with the results of \cite{WZ2,DJ,LQS} where they proved that all fractional Yamabe constants with $\gamma\in(0,1]$ are lower bounds of the relative volume for Poincar\'{e}-Einstein manifold, it seems that the family of fractional  GJMS operators shares many intrinsic properties to be discovered (e.g. the monotonicity of the fractional Yamabe constants). Here we pay more attention to the comparison inequality for general fractional GJMS operators. Our first result is the comparison theorem between $P^{\hat{g}}_{1}$ and $P^{\hat{g}}_{2\gamma}$.

\begin{theorem}\label{mth1}
Suppose $(X^{n+1},g_+)$ $(n\geq 3)$ is a Poincar\'{e}-Einstein manifold with conformal infinity $(M,[\hat{g}])$. Assume $\lambda_1 (-\Delta_+)>\frac{n^2 -1}{4}$ and $Y_1(\partial X, [\hat{g}])\geq 0$ can be achieved by some smooth representative $\hat{g}$.
Then for all $\gamma\in \left(\frac{1}{2},1\right)$, 
\begin{equation}\label{ineq1}
\frac{Y_{2\gamma}(M,[\hat{g}])}{Y_{2\gamma}(\mathbb{S}^n,[g_{\mathbb{S}}])}
 \leq \left(\frac{Y_1(M,[\hat{g}])}{Y_1(\mathbb{S}^n,[g_{\mathbb{S}}])}\right)^{2\gamma},
\end{equation}
and the equality holds if and only if $(X, g_+)$ is isometric to the hyperbolic space $\mathbb{H}^{n+1}$.
\end{theorem}

Notice that $\lambda_1 (-\Delta_+)>\frac{n^2 -1}{4}> \frac{n^2 -(2\gamma)^2}{4}$ for $\gamma\in \left(\frac{1}{2},1\right)$, which satisfies the spectrum condition of \cite[Section 6]{CC}. Our next main result is the comparison theorem between $P^{\hat{g}}_{2}$ and $P^{\hat{g}}_{2\gamma}$.

\begin{theorem}\label{mth2}
Suppose $(X^{n+1},g_+)$ $(n\geq 3)$ is a Poincar\'{e}-Einstein manifold with conformal infinity $(M,[\hat{g}])$ of nonnegative Yamabe type. 
Then for all $\gamma\in \left(1,\min\{\frac{n}{2},2\}\right)$, 
\begin{equation}\label{ineq2}
\frac{Y_{2\gamma}(M,[\hat{g}])}{Y_{2\gamma}(\mathbb{S}^n,[g_{\mathbb{S}}])}
 \leq \left(\frac{Y_2(M,[\hat{g}])}{Y_2(\mathbb{S}^n,[g_{\mathbb{S}}])}\right)^{\gamma},
\end{equation}
and the equality holds if and only if $(X, g_+)$ is isometric to the hyperbolic space $\mathbb{H}^{n+1}$.
\end{theorem}

Together with Theorem \ref{mth0}, Theorem \ref{mth1} and Theorem\ref{mth2}, we may partially give some evidence for the monotonicity of the fractional Yamabe constants.


\vspace{0.1in}
In this paper, we need a strong regularity assumption on the conformal compactification of $g_+$, i.e. $\rho^2 g_+$ extends smoothly to the boundary. The main reason is to make sure that the fractional GJMS operators can be well defined, see \cite{MM, JS, CDLS, CC}. For the proof of Theorem \ref{mth0}, \cite{CLW}, \cite{WZ1} and \cite{WW} are based on the idea of Gursky-Han\cite{GH} where they carefully deal with an integral term for some compact compactificated metric, i.e. $\bar{g}=\rho^2 g_+$, and 
\[
\int_{X}\rho |E_{\bar{g}}|^{2}_{\bar{g}} dv_{\bar{g}},
\]
where $E_{\bar{g}}$ is the trace-free Ricci for $\bar{g}$ and
\[
E_{\bar{g}}=-(n-1)\rho\left(\bar{\nabla}^{2} \rho-\frac{\Delta_{\bar{g}}\rho}{n+1}\bar{g}\right).
\]
However, for general $\gamma$, such method does not work well. We return to study the unique solution $u$ of the Poisson equation (\ref{pos}) with $f=1$ for the hyperbolic space. We find out that some properties of the solution $u$ still hold for Poincar\'{e}-Einstein manifold, especially the relationships of different solutions in terms of different $\gamma$. Then we can construct the needed test defining function according to the above properties and relationships. For more details, one can refer to Section \ref{model}. 

\vspace{0.1in} 
The paper is outlined as follows: In section \ref{model}, we mainly study the unique solution of the Poisson equation (\ref{pos}) for the hyperbolic space and their relationships in terms of different $\gamma$ which can be used to construct our test defining function. In section \ref{y1} and \ref{y2}, we prove Theorem \ref{mth1} and Theorem \ref{mth2} respectively. Indeed, we obtain the point-wise upper bounds for $Q_{2\gamma}$ under our assumptions which can directly imply our main theorems.

\vspace{0.1in}
\textbf{Acknowledgment:} This work was partially supported by the Fundamental Research Funds for the Central Universities (Grant No. WK0010000081).The author would like to thank professor Fang Wang and Zuoqin Wang for their valuable suggestions.

\vspace{0.2in}
\section{Hyperbolic Space $\mathbb{H}^{n+1}$}\label{model}

In this section, we mainly consider the hyperbolic space $\mathbb{H}^{n+1}(n\geq 3)$ to discover some interesting properties. Here the hyperbolic space is given by
\[
(\mathbb{H}^{n+1},g_{\mathbb{H}})=\left(\mathbb{B}^{n+1},\frac{4dz^2}{(1-|z|^2)^2}\right)
\]
with conformal infinity $(\mathbb{S}^{n},[g_{\mathbb{S}}])$, where $g_{\mathbb{S}}$ is the round metric on $\mathbb{S}^n$. Then the corresponding geodesic normal defining function is
\[
r=\frac{2(1-|z|)}{1+|z|}\in [0,2].
\]

Recall $s=\frac{n}{2}+\gamma$, let $u^{\mathbb{H}}_s $ be the unique solution satisfying the following equation:
\begin{equation}\label{h-u}
-\Delta_{\mathbb{H}}u^{\mathbb{H}}_{s}-s(n-s)u^{\mathbb{H}}_{s}=0, \quad r^{s-n}u^{\mathbb{H}}_{s}|_{\mathbb{S}^n}=1.
\end{equation}
In particular, we can solve equation (\ref{h-u}) and obtain that $u^{\mathbb{H}}_s$ satisfies 
\begin{itemize}
\item[(i)]  $u^{\mathbb{H}}_s$ is a smooth positive usual Hypergeometric function of $r$;
\item[(ii)] $u^{\mathbb{H}}_s$ is monotone increasing in $r$;
\item[(iii)] $u^{\mathbb{H}}_s$ has the asymptotics
\[
u^{\mathbb{H}}_{s}=r^{n-s}\left(1+ u^{\mathbb{S}}_2 r^2 +u^{\mathbb{S}}_4 r^4+o(r^4)\right)+r^s \left(u^{\mathbb{S}}_{2\gamma}+O(r^2)\right),\quad \textrm{as } r\to 0,
\]
where
\[
u^{\mathbb{S}}_2 =-\frac{n-2\gamma}{4(2\gamma-2)}J^{\mathbb{S}}, \quad  u^{\mathbb{S}}_{2\gamma}=\frac{n-2\gamma}{2d_{\gamma}}Q_{2\gamma}^{\mathbb{S}},
\]
\[
\begin{aligned}
u^{\mathbb{S}}_4
&=\frac{n-2\gamma}{16(2\gamma-2)(2\gamma-4)}\left(-\Delta_{g_{\mathbb{S}}}J^{\mathbb{S}}+\frac{n-2\gamma+4}{2}(J^{\mathbb{S}})^2 -(2\gamma-2)|A^{\mathbb{S}}|^{2}_{g_{\mathbb{S}}}\right) \\
&=\frac{n-2\gamma}{16(2\gamma-2)(2\gamma-4)}\left(\frac{n-2\gamma+4}{2}-\frac{2\gamma-2}{n}\right)(J^{\mathbb{S}})^2.
\end{aligned}
\]
Here $J^{\mathbb{S}}=J_{g_\mathbb{S}}=\frac{n}{2}$, $Q_{2\gamma}^{\mathbb{S}}=Q_{2\gamma}^{g_\mathbb{S}}$ and $A^{\mathbb{S}}=A_{g_{\mathbb{S}}}$ for short.
\end{itemize}
For more details, one can refer to \cite{Go,WZ2}.

\vspace{0.2in}
\subsection{Case $s_1 =\frac{n+1}{2}$} When $s=s_1 =\frac{n+1}{2}$ in equation (\ref{h-u}), we can obtain
\[
u^{\mathbb{H}}_{s_1}=\left(\frac{4r}{(2+r)^2}\right)^{\frac{n-1}{2}}=\left(\frac{1-|z|^2}{2}\right)^{\frac{n-1}{2}}. 
\]
Let $\rho^{\mathbb{H}}_{s_1}:=(u^{\mathbb{H}}_{s_1})^{\frac{2}{n-1}}$, then
\[
\rho^{\mathbb{H}}_{s_1}=\frac{4r}{(2+r)^2}=r\left(1- r+\frac{3}{4} r^2-\frac{1}{2} r^3 +O(r^4)\right).
\]
It is easy to see that $\rho^{\mathbb{H}}_{s_1}$ is monotone increasing in $r$ and $\rho^{\mathbb{H}}_{s_1}\in \left[0,\frac{1}{2}\right]$. Moreover, 
\[
(\rho^{\mathbb{H}}_{s_1})^{-2}|\nabla \rho^{\mathbb{H}}_{s_1}|_{g_{\mathbb{H}}}^2=1-2\rho^{\mathbb{H}}_{s_1}. 
\]

Take $\rho=\rho^{\mathbb{H}}_{s_1}$ for simplicity. For general $\gamma\in \left(0,1\right)$ and $s=\frac{n}{2}+\gamma$. Define
\begin{equation}\label{dphi}
\varphi(\rho):=\frac{u^{\mathbb{H}}_s }{(u^{\mathbb{H}}_{s_1})^{\frac{n-2\gamma}{n-1}}}=u^{\mathbb{H}}_s \cdot \rho^{-\frac{n-2\gamma}{2}},\quad \rho\in \left[0,\frac{1}{2}\right].
\end{equation}
Then $\varphi(\rho)$ is a smooth function of $\rho$ and $\varphi(0)=1$. $\varphi(\rho)$ also has asymptotics for $\gamma\in (1/2,1)$,
\[
\varphi(\rho)=1+\varphi_{1}\rho+\varphi_{2\gamma}\rho^{2\gamma}+\varphi_{2} \rho^2 + o\left(\rho^{2}\right),\quad \textrm{as }\rho\to 0,
\]
where
\[
\varphi_{1}=\frac{n-2\gamma}{2},\quad
\varphi_{2\gamma}=\frac{n-2\gamma}{2d_{\gamma}}Q_{2\gamma}^{g_{\mathbb{S}}},
\]
\[
\varphi_{2}=\frac{n-2\gamma}{8}+\frac{(n-2\gamma)(n-2\gamma-2)}{8}+u^{\mathbb{S}}_2,
\quad \textrm{with }
u^{\mathbb{S}}_2 =-\frac{n-2\gamma}{4(2\gamma-2)}J^{\mathbb{S}}.
\]
Moreover, $\varphi(\rho)$ satisfies
\begin{equation}\label{phiab}
\Delta_{g_\mathbb{H}}\left(\varphi(\rho)\rho^{\frac{n-2\gamma}{2}}\right)+\frac{n^2 -4\gamma^2}{4}\left(\varphi(\rho)\rho^{\frac{n-2\gamma}{2}}\right)=\rho^{\frac{n-2\gamma}{2}}\left( \bar{A}_{\gamma}(\rho)\rho^{-2}|d\rho|^{2}_{g_\mathbb{H}}+\bar{B}_{\gamma}(\rho)\right)=0,
\end{equation}
where
\[
\begin{aligned}
&\bar{A}_{\gamma}(\rho):=\frac{n-2\gamma}{2}\frac{1-2\gamma}{2}\varphi(\rho)+\left(\frac{n+3}{2}-2\gamma\right)\rho\varphi'(\rho)+\rho^{2}\varphi''(\rho),\\
&\bar{B}_{\gamma}(\rho):=\frac{n-2\gamma}{2}\frac{2\gamma-1}{2}\varphi(\rho)-\frac{n+1}{2}\rho\varphi'(\rho).\\
\end{aligned}
\]

\vspace{0.1in}
\begin{lemma}\label{bpos}
For $0<\rho<\frac{1}{2}$, we have
\begin{itemize}
\item if $\frac{1}{2}<\gamma<1$, then $\bar{B}_{\gamma}(\rho)>0$, and $\bar{A}_{\gamma}(\rho)<0$.
\item if $0<\gamma<\frac{1}{2}$, then $\bar{B}_{\gamma}(\rho)<0$, and $\bar{A}_{\gamma}(\rho)>0$.
\end{itemize}
\end{lemma}

\begin{proof}
Here we only deal with case $\frac{1}{2}<\gamma<1$ since the proof of the other case is similar. Firstly recall 
\[
\rho^{-2}|d \rho|^{2}_{g_\mathbb{H}}=1-2\rho >0, \quad \textrm{as }\rho\in [0,1/2].
\]
Then $\bar{B}_{\gamma}(\frac{1}{2})=0$ according to equation (\ref{phiab}). Notice that $\varphi(0)=1$, we have
\[
\bar{B}_{\gamma}(0)=\frac{n-2\gamma}{2}\frac{2\gamma-1}{2}\varphi(0)=\frac{n-2\gamma}{2}\frac{2\gamma-1}{2}>0.
\]
If we suppose that $\bar{B}_{\gamma}(\rho)>0$ for $\rho\in (0,1/2)$ is not true, we must have $\bar{B}_{\gamma}(\rho)=0$ for some $0<\rho<\frac{1}{2}$. Assume that  $\rho^* $ is the first zero point of $\bar{B}_{\gamma}(\rho)$ with $0<\rho^*<\frac{1}{2}$. Then $\bar{B}_{\gamma}(\rho^*)=0$ and $\bar{B}'_{\gamma}(\rho^*)\leq 0$. Recall 
\[
\rho^{-2}|d \rho|^{2}_{g_\mathbb{H}}=1-2\rho >0, \quad \textrm{as }\rho\in (0,1/2).
\]
Then we know that $\bar{A}_{\gamma}(\rho^*)$ is also vanishing from equation (\ref{phiab}). Now we have
\begin{equation}\label{abe1}
\begin{aligned}
&\bar{A}_{\gamma}(\rho^*)=\frac{n-2\gamma}{2}\frac{1-2\gamma}{2}\varphi (\rho^*)+\left(\frac{n+3}{2}-2\gamma\right)\rho^* \varphi'(\rho^*)+(\rho^{*})^2\varphi''(\rho^*)=0,\\
&\bar{B}_{\gamma}(\rho^*)=\frac{n-2\gamma}{2}\frac{2\gamma-1}{2}\varphi(\rho^*)-\frac{n+1}{2}\rho^* \varphi'(\rho^*)=0.\\
\end{aligned}
\end{equation}
We compute the derivative of $\bar{B}_{\gamma}(\rho)$ at $\rho^*$, which satisfies
\[
\rho^* \bar{B}_{\gamma}'(\rho^*)=\frac{n-2\gamma}{2}\frac{2\gamma-1}{2}\rho^* \varphi'(\rho^*)-\frac{n+1}{2}\rho^* \varphi'(\rho^*)-\frac{n+1}{2}(\rho^*)^2\varphi''(\rho^*).
\]
Replacing $\varphi''(\rho^*)$ and $\varphi'(\rho^*)$ in the above equation according to equation (\ref{abe1}), we can obtain
\[
\rho^* \bar{B}_{\gamma}'(\rho^*)=-\frac{(n-2\gamma)(2\gamma-1)(n+2\gamma)(2\gamma+1)}{8(n+1)}\varphi(\rho^*)<0.
\]
Then by continuity of $\bar{B}_{\gamma}(\rho)$ and $\bar{B}_{\gamma}(\frac{1}{2})=\bar{B}_{\gamma}(\rho^*)=0$, there exists $\rho^{**}\in \left(\rho^*, \frac{1}{2}\right)$ such that 
\[\bar{B}'_{\gamma}(\rho^{**})=0,\quad \textrm{and}\quad \bar{B}_{\gamma}(\rho^{**})<0,\]
which are equivalent to respectively
\begin{equation}\label{**1}
\left(\frac{n-2\gamma}{2}\frac{2\gamma-1}{2}-\frac{n+1}{2}\right)\rho^{**} \varphi'(\rho^{**})-\frac{n+1}{2}(\rho^{**})^2\varphi''(\rho^{**})=0,
\end{equation}
and
\begin{equation}\label{**2}
\frac{n+1}{2}\rho^{**}\varphi'(\rho^{**})>\frac{(n-2\gamma)(2\gamma-1)}{4}\varphi(\rho^{**})>0.
\end{equation}
Recall equation (\ref{phiab}) and $\rho^{-2}|d \rho|^{2}_{g_\mathbb{H}}=1-2\rho$,  
\begin{equation}\label{**3}
\bar{A}_{\gamma}(\rho^{**})(1-2\rho^{**})+\bar{B}_{\gamma}(\rho^{**})=0.
\end{equation}
Inserting equation (\ref{**1}) into (\ref{**3}) with the definition of $\bar{A}_{\gamma}(\rho)$ and $\bar{B}_{\gamma}(\rho)$ to replace $\varphi''(\rho^{**})$, we have
\begin{equation}
\frac{(n-2\gamma)(1-2\gamma)}{2}\varphi(\rho^{**})=\left((1-2\rho^{**})\left(\frac{n+1}{2}-2\gamma+\frac{(n-2\gamma)(2\gamma-1)}{2(n+1)}\right)-\frac{n+1}{2}\right)\varphi'(\rho^{**}).
\end{equation}
Then together with equation (\ref{**2}), we have
\[
-(n+1)\rho^{**}<(1-2\rho^{**})\left(\frac{n+1}{2}-2\gamma+\frac{(n-2\gamma)(2\gamma-1)}{2(n+1)}\right)-\frac{n+1}{2}.
\]
Combining the coefficient of $\rho^{**}$ on both sides, we have
\[
\frac{(n+2\gamma)(2\gamma+1)}{n+1}\rho^{**}>\frac{(n+2\gamma)(2\gamma+1)}{2(n+1)}.
\]
Therefore
\[
\rho^{**}>\frac{1}{2}
\]
which is a contradiction with $\rho^* <\rho^{**}<\frac{1}{2}$. Thus we conclude that $\bar{B}_{\gamma}(\rho)$ has no zero point  when $0<\rho<\frac{1}{2}$, i.e.
\[
\bar{B}_{\gamma}(\rho)>0 \quad \textrm{as } 0<\rho<\frac{1}{2}.
\]
By equation (\ref{phiab}) and $\rho^{-2}|d \rho|^{2}_{g_\mathbb{H}}=1-2\rho>0$ again, we have
\[
\bar{A}_{\gamma}(\rho)<0 \quad \textrm{as } 0<\rho<\frac{1}{2}.
\]
\end{proof}

\vspace{0.1in}
\subsection{Case $s_2 =n+1$} When $s = s_2 =n+1$ in equation (\ref{h-u}), we can have
\[
u^{\mathbb{H}}_{s_2}=\frac{4+r^2}{4r}=\frac{1+|z|^2}{1-|z|^2}.
\]
Let $\rho^{\mathbb{H}}_{s_2}:=(u^{\mathbb{H}}_{s_2})^{-1}$, then
\[
\rho^{\mathbb{H}}_{s_2}=\frac{4r}{4+r^2}=r\left(1-\frac{1}{4} r^2 +\frac{1}{16} r^4+o(r^4)\right).
\]
It is easy to see that $\rho^{\mathbb{H}}_{s_2} \in [0,1]$ is also increasing in $r$. Moreover $\rho^{\mathbb{H}}_{s_2}$ satisfies
\[
(\rho^{\mathbb{H}}_{s_2})^{-2} |\nabla \rho^{\mathbb{H}}_{s_2}|_{g_{\mathbb{H}}}^2=1-(\rho^{\mathbb{H}}_{s_2})^2.
\]

Similarly take $\rho=\rho^{\mathbb{H}}_{s_2}$ for short. For general $\gamma\in \left(0,\min \{\frac{n}{2},2\}\right)$ and $s=\frac{n}{2}+\gamma$, define
\begin{equation}\label{dpsi}
\psi(\rho):=\frac{u^{\mathbb{H}}_s }{(u^{\mathbb{H}}_{s_2})^{-\frac{n-2\gamma}{2}}}=u^{\mathbb{H}}_s \cdot \rho^{-\frac{n-2\gamma}{2}}.
\end{equation}
Then $\psi(\rho)$ is a positive smooth function of $\rho$ and $\psi(0)=1$. $\psi(\rho)$ has the following asymptotic for $\gamma\in (1,\min \{n/2,2\})$,
\[
\psi(\rho)=1+\psi_2 \rho^2 +\psi_{2\gamma}\rho^{2\gamma}+\psi_4 \rho^4 +o(\rho^{4}),\quad \textrm{as }\rho\to 0,
\]
where
\[
\begin{aligned}
\psi_2 &=\frac{n-2\gamma}{8} +u^{\mathbb{S}}_2 ,\quad \psi_{2\gamma}=\frac{n-2\gamma}{2d_{\gamma}}Q_{2\gamma}^{\mathbb{S}},\\
\psi_4 &=-\frac{n-2\gamma}{32} +\frac{(n-2\gamma)(n-2\gamma+10)}{128}+\frac{n-2\gamma+4}{8} u^{\mathbb{S}}_2 +u^{\mathbb{S}}_4.\\
\end{aligned}
\]
Here 
\[
u^{\mathbb{S}}_2 =-\frac{n-2\gamma}{4(2\gamma-2)}J^{\mathbb{S}}, 
\quad
u^{\mathbb{S}}_4=\frac{n-2\gamma}{16(2\gamma-2)(2\gamma-4)}\left(\frac{n-2\gamma+4}{2}-\frac{2\gamma-2}{n}\right)(J^{\mathbb{S}})^2.
\]
Moreover, $\psi(\rho)$ satisfies
\begin{equation}\label{psi-ef}
\Delta_{\mathbb{H}}\left(\psi(\rho)\rho^{\frac{n-2\gamma}{2}}\right)+\frac{n^2 -4\gamma^2}{4}\left(\psi(\rho)\rho^{\frac{n-2\gamma}{2}}\right)=\rho^{\frac{n-2\gamma}{2}}\left( \tilde{A}_{\gamma}(\rho)\rho^{-2}|d\rho|^{2}_{\mathbb{H}}+\tilde{B}_{\gamma}(\rho)\right)=0,
\end{equation}
where
\[
\begin{aligned}
&\tilde{A}_{\gamma}(\rho):=\rho^2 \psi''(\rho)+(n-2\gamma+2)\rho\psi'(\rho)+\frac{(n-2\gamma)(n-2\gamma+2)}{4}\varphi(\rho),\\
&\tilde{B}_{\gamma}(\rho):=-(n+1)\rho\psi'(\rho)-\frac{(n-2\gamma)(n-2\gamma+2)}{4}\psi(\rho).\\
\end{aligned}
\]

\vspace{0.1in}
\begin{lemma}\label{fneg}
If $\gamma\in \left(0,\min \{\frac{n}{2},2\}\right)$, then
\[
\tilde{A}_{\gamma}(\rho)>0,\quad \tilde{B}_{\gamma}(\rho)<0, \quad \textrm{as } 0<\rho<1.
\]
\end{lemma}

\begin{proof}

First note that $\psi(0)=1$, then
\[
\tilde{B}_{\gamma}(\rho)=-\frac{(n-2\gamma)(n-2\gamma+2)}{4}\psi(0)=-\frac{(n-2\gamma)(n-2\gamma+2)}{4}<0,
\]
Recall equation (\ref{psi-ef}) and $\rho^{-2}|d \rho|^{2}_{\mathbb{H}}=1-\rho^2$, then $\tilde{B}_{\gamma}(1)=0$. If $\tilde{B}_{\gamma}(\rho)<0$ for $\rho\in (0,1)$ is not true, we must have $\tilde{B}_{\gamma}(\rho)=0$ for some $0<\rho<1$. Assume that $\rho^* \in (0,1)$ is the first zero point of $\tilde{B}_{\gamma}$. Then $\tilde{B}_{\gamma}(\rho^*)=0$, and $\tilde{B}_{\gamma}'(\rho^*)\geq 0$.  Recall 
\[
\rho^{-2}|d \rho|^{2}_{\mathbb{H}}=1-\rho^2 >0, \quad \textrm{as }\rho\in (0,1).
\]
Then $\tilde{A}_{\gamma}(\rho^*)=0$ by equation (\ref{psi-ef}) again.  
Thus we have
\begin{equation}\label{ef-1}
\begin{aligned}
&\tilde{A}_{\gamma}(\rho^*)=(\rho^*)^2 \psi''(\rho^*)+(n-2\gamma+2)\rho^*\psi'(\rho^*)+\frac{(n-2\gamma)(n-2\gamma+2)}{4}\psi(\rho^*)=0,\\
&\tilde{B}_{\gamma}(\rho^*)=-(n+1)\rho^* \psi'(\rho^*)-\frac{(n-2\gamma)(n-2\gamma+2)}{4}\psi(\rho^*)=0.\\
\end{aligned}
\end{equation}
By direct computation, the derivative of $\tilde{B}_{\gamma}(\rho)$ at $\rho^*$ satisfies
\[
\rho^* \tilde{B}_{\gamma}'(\rho^*)=-(n+1)(\rho^*)^2 \psi''(\rho^*)-(n+1)\rho^* \psi'(\rho^*)-\frac{(n-2\gamma)(n-2\gamma+2)}{4}\rho^*\psi'(\rho^*).
\]
Then we use equation (\ref{ef-1}) to replace $\psi''(\rho^*)$ and $\psi'(\rho^*)$ in the above equation,  
\[
\rho^* \tilde{B}_{\gamma}'(\rho^*)=\frac{(n-2\gamma)(n-2\gamma+2)}{4}\left(2\gamma+\frac{(n-2\gamma)(n-2\gamma+2)}{4(n+1)}\right)\psi(\rho^*)>0.
\]
Because of $\tilde{B}_{\gamma}(1)=\tilde{B}_{\gamma}(\rho^*)=0$ and the continuity of $\tilde{B}_{\gamma}(\rho)$, there exists $\rho^{**}\in (\rho^*, 1)$ such that 
\[\tilde{B}_{\gamma}'(\rho^{**})=0,\quad \tilde{B}_{\gamma}(\rho^{**})>0.\]
It follows that
\begin{equation}\label{***1}
\rho^{**} \tilde{B}_{\gamma}'(\rho^{**})=-(n+1)(\rho^{**})^2 \psi''(\rho^{**})-\left(n+1+\frac{(n-2\gamma)(n-2\gamma+2)}{4}\right)\rho^{**}\psi'(\rho^{**})=0,
\end{equation}
and
\begin{equation}\label{***2}
(n+1)\rho^{**}\psi'(\rho^{**})<-\frac{(n-2\gamma)(n-2\gamma+2)}{4}\psi(\rho^{**})<0.
\end{equation}
Recall equation (\ref{psi-ef}) and $\rho^{-2}|d \rho|^{2}_{\mathbb{H}}=1-\rho^2$, 
\begin{equation}\label{***3}
\tilde{A}_{\gamma}(\rho^{**})\left(1-(\rho^{**})^2\right)+\tilde{B}_{\gamma}(\rho^{**})=0.
\end{equation}
Inserting equation (\ref{***1}) into (\ref{***3}), we have
\begin{multline}
\frac{(n-2\gamma)(n-2\gamma+2)}{4}\rho^{**}\psi(\rho^{**})=\\
\left(\left (1-(\rho^{**})^2\right)\left(n-2\gamma+1-\frac{(n-2\gamma)(n-2\gamma+2)}{4(n+1)}\right)-(n+1)\right)\psi'(\rho^{**}).
\end{multline}
Then according to (\ref{***2}), we get
\[
-(n+1)(\rho^{**})^2 < \left (1-(\rho^{**})^2\right)\left(n-2\gamma+1-\frac{(n-2\gamma)(n-2\gamma+2)}{4(n+1)}\right)-(n+1),
\]
which is equivalent to
\[
\left(-2\gamma -\frac{(n-2\gamma)(n-2\gamma+2)}{4(n+1)}\right)\left(1-(\rho^{**})^2\right) >0.
\]
It follows that $\rho^{**}>1$ since
\[
-2\gamma -\frac{(n-2\gamma)(n-2\gamma+2)}{4(n+1)}\leq 0\quad \textrm{as } \gamma\in \left(0,\min \{\frac{n}{2},2\}\right).
\]
Thus we obtain a contradiction with $\rho^*<\rho^{**}<1$. Therefore $\tilde{B}_{\gamma}(\rho)$ has no zero point as $0<\rho<1$, i.e.
\[
\tilde{B}_{\gamma}(\rho)<0,\quad \textrm{as } 0<\rho<1.
\]
By equation (\ref{psi-ef}) and $\rho^{-2}|d \rho|^{2}_{\mathbb{H}}=1-\rho^2>0$ again, we obtain
\[
\tilde{A}_{\gamma}(\rho)>0,\quad \textrm{as } 0<\rho<1.
\]
\end{proof}

\vspace{0.2in}
\section{$Y_{1} (M,[\hat{g}])$ vs. $Y_{2\gamma} (M,[\hat{g}])$}\label{y1}
In this section, for Poincar\'{e}-Einstein manifold $(X^{n+1},g_+)(n\geq 3)$ with conformal infinity $(M,[\hat{g}])$, we assume that $Y_1 (M,[\hat{g}])\geq 0$ is attained by $\hat{g}$ and $r$ is the geodesic defining function for $\hat{g}$. Take $Q_{2\gamma} = Q_{2\gamma}^{\hat{g}}$ and $\hat{J}=J_{\hat{g}}$ for short.

\begin{itemize}
\item Let $\bar{v}$ satisfy the linear equation
\[\Delta_+ \bar{v} +\frac{n^2 -1}{4}\bar{v}=0,\quad \bar{v}\sim r^{\frac{n-1}{2}} \textrm{ as }r\to 0.\]
Then $\bar{v}$ has asymptotics
\[
\begin{aligned}
&\bar{v}=r^{\frac{n-1}{2}}(1+\bar{v}_1 r +\bar{v}_2 r^2 +\bar{v}_3 r^3+ O(r^4)),\\
&\bar{v}_1 = -\frac{n-1}{2}Q_1,\\
&\bar{v}_2 = \frac{n-1}{4}\hat{J}.\\
\end{aligned}
\]
\item Let $\bar{\rho}=\bar{v}^{\frac{2}{n-1}}$, then
\[
\begin{aligned}
&\bar{\rho}=r(1+\bar{\rho}_1 r +\bar{\rho}_2 r^2 +\bar{\rho}_3 r^3+ O(r^4)),\\
&\bar{\rho}_1 = -Q_1,\\
&\bar{\rho}_2 = \frac{1}{2}\hat{J}-\frac{n-3}{4}Q_{1}^2.\\
\end{aligned}
\]
\end{itemize}

Let $\bar{g}=\bar{\rho}^2 g_+$, 
define
$$\bar{T}:=\frac{1-|d\bar{\rho}|_{\bar{g}}^2}{\bar{\rho}}.$$
Then
\[
\Delta_{\bar{g}}\bar{\rho}=-\frac{n+1}{2}\bar{T}.
\]
\begin{lemma}\label{btlem}
$\bar{T}$ satisfies 
\begin{equation}\label{bte}
\Delta_{\bar{g}}\bar{T}=-\frac{2}{(n-1)^2}\bar{\rho} |\bar{E}|_{\bar{g}}^2,\quad \bar{T}|_{M}=2 Q_1.
\end{equation}
where $\bar{E}=E_{\bar{g}}$ for short and $\bar{T}$ has asymptotics
\[
\bar{T}=2Q_1 +(nQ_{1}^2 -2\hat{J})r+O(r^2),\quad \textrm{as }r\to 0.
\]
Moreover, if $Q_1$ is constant, then 
\[
\bar{T}\geq 2Q_1, \quad \textrm{and}\quad \hat{J}\leq \frac{n}{2}Q_{1}^2.
\]
Either one of the equalities holds if and only if $(X^{n+1},g_+)$ is isometric to the hyperbolic space.
\end{lemma}

\begin{proof}
By direct calculation, we can have equation (\ref{bte}) and the asymptotics of $\bar{T}$. If $Q_1$ is constant, the maximum principle tells us that $\bar{T}\geq 2Q_1$ and the Hopf  Lemma implies that
\[
nQ_{1}^2 -2\hat{J}\geq 0.
\]

If the equality holds, i.e. $nQ_{1}^2 -2\hat{J}= 0$. We integrate over $X$ for both side of equation (\ref{bte}) and obtain
\[
-\frac{2}{(n-1)^2}\int_{X}\bar{\rho} |\bar{E}|_{\bar{g}}^2 ~dv_{\bar{g}}=\int_{X}\Delta_{\bar{g}}\bar{T}dv_{\bar{g}}=\int_{M}\frac{\partial \bar{T}}{\partial \bar{\nu}} dv_{\hat{g}}=\int_{M} -(nQ_{1}^2 -2\hat{J})dv_{\hat{g}}=0.
\]
where $\bar{\nu}$ is the outer unit normal vector for $\bar{g}$.
Therefore we must have $\bar{E}=0$ and $\bar{T}=2Q_1$. It is easy to compute that
\[
\bar{E}=-(n-1)\bar{\rho}\left(\bar{\nabla}^{2} \bar{\rho}-\frac{\Delta_{\bar{g}}\bar{\rho}}{n+1}\bar{g}\right),\quad
\Delta_{\bar{g}}\bar{\rho}=-\frac{n+1}{2}\bar{T}=-(n+1)Q_1
\] 
which implies that
\[
\bar{\nabla}^{2} \bar{\rho}+Q_1 \bar{g}=0,\quad \bar{\rho}|_M =0,
\] 
where $\bar{\nabla}^{2}$ is the Hessian operator for $\bar{g}$.
Then according to the standard argument of Obata equation or \cite[Lemma 2.2]{WZ1}, $Q_1$ must be a positive constant and $(X,\bar{g})$ is isometric to the Euclidean ball, i.e. $(X^{n+1},g_+)$ must be isometric to the hyperbolic space.
\end{proof}

\begin{lemma}\label{q1null}
Let $\gamma\in (\frac{1}{2},1)$, if $Q_1 =0$, then
\[
Q_{2\gamma}\leq 0.
\]
Moreover, there exists a point $p\in M$ such that  $Q_{2\gamma}(p)<0$.
\end{lemma}
\begin{proof}
Let $\bar{w}$ satisfy
\[
\Delta_+ \bar{w}+\frac{n^2 -(2\gamma)^2}{4}\bar{w} = 0,\quad \bar{w}\sim r^{\frac{n-2\gamma}{2}} \textrm{ as }r\to 0.
\]
Then $\bar{w}$ has the asymptotics 
\[
\bar{w}=r^{\frac{n-2\gamma}{2}}\left(1+ \bar{w}_{2\gamma}r^{2\gamma}+\bar{w}_2 r^2 +o(r^2)\right),
\]
where
\[
\begin{aligned}
\bar{w}_{2\gamma}&=\frac{n-2\gamma}{2d_{\gamma}}Q_{2\gamma},\\
\bar{w}_2 &=-\frac{n-2\gamma}{4(2\gamma-2)}\hat{J}.\\
\end{aligned}
\]
Recall that $\bar{v}$ satisfies 
\[
\Delta_+ \bar{v}+\frac{n^2 -1}{4}\bar{v}=0,
\]
and 
\[
\bar{v}=r^{\frac{n-1}{2}}\left(1 +\frac{n-1}{4}\hat{J}r^2 +o(r^2)\right),
\]
since $Q_1 =0$. Then $\bar{g}=\bar{\rho}^{2}g_+$ with $\bar{\rho}=\bar{v}^{\frac{2}{n-1}}$. According to Lemma \ref{btlem}, we have
\[1-\frac{4}{(n-1)^2}\bar{v}^{-2}|d \bar{v}|_{+}^2=1-|d\bar{\rho}|_{\bar{g}}^2 \geq 0.\]
Define
\[
\bar{\xi}:= \bar{v}^{\frac{n-2\gamma}{n-1}}.
\]
Then
\[
\bar{\xi}=r^{\frac{n-1}{2}}\left(1+ \frac{n-2\gamma}{4}\hat{J} r^2 +o(r^2)\right).
\]
By direct computation, we have
\[
\Delta_+ \bar{\xi}+\frac{n^2 -(2\gamma)^2}{4}\bar{\xi} =\frac{(n-2\gamma)(2\gamma-1)}{4}\bar{\xi} \left(1-\frac{4}{(n-1)^2}\bar{v}^{-2}|d \bar{v}|_{+}^2 \right)\geq 0.
\]
Note that $\frac{\bar{w}}{\bar{\xi}}$ satisfies the equation:
\[
-\Delta_+ \left(\frac{\bar{w}}{\bar{\xi}}\right)=\left(\frac{n^2 -(2\gamma)^2}{4}+\frac{\Delta_+ \bar{\xi}}{\bar{\xi}}\right)\frac{\bar{w}}{\bar{\xi}}+2\left\langle d\left(\frac{\bar{w}}{\bar{\xi}}\right),\frac{d \bar{w}}{\bar{w}}\right\rangle_{g_+},\quad \left(\frac{\bar{w}}{\bar{\xi}}\right)\big{|}_{M}=1.
\]
Notice that $\frac{\bar{w}}{\bar{\xi}}>0$. Applying maximum principle, $\frac{\bar{w}}{\bar{\xi}}\geq 1$, i.e. $\bar{w}\geq \bar{\xi}$. Thus near the boundary, 
\[
1+\bar{w}_{2\gamma}r^{2\gamma} +\bar{w}_2 r^2 +o(r^2)\geq 1 +\frac{n-2\gamma}{4}\hat{J}  r^2 +o(r^2).
\]
Therefore $\bar{w}_{2\gamma}\geq 0$ which implies that $Q_{2\gamma}\leq 0$ since $d_\gamma<0$.

At last, if $Q_{2\gamma}\equiv0$, then we have
\[
\bar{w}_2 =-\frac{n-2\gamma}{4(2\gamma-2)}\hat{J} \geq \frac{n-2\gamma}{4}\hat{J},
\]
which is equivalent to
\[
\frac{(n-2\gamma)(2\gamma-1)}{4(2-2\gamma)}\hat{J}\geq 0.
\]
Thus $\hat{J}\geq 0$. However, by Lemma \ref{btlem}, $\hat{J}\leq \frac{n}{2}Q_{1}^2 =0$. Therefore we have
\[
\hat{J}=\frac{n}{2}Q_{1}^2=0.
\]
By the proof of Lemma \ref{btlem} again, we obtain $Q_1$ is a positive constant which is a contradiction. Thus there exists a point $p\in M$ such that  $Q_{2\gamma}(p)<0$.
\end{proof}

\begin{proposition}\label{qpn1}
Let $\gamma\in (\frac{1}{2},1)$, if $Q_1$ is a positive constant, then
\[
\frac{Q_{2\gamma}}{Q_{2\gamma}^{\mathbb{S}}}\leq \frac{Q_{1}}{Q_{1}^{\mathbb{S}}}.
\]
The equality holds if and only if $(X^{n+1},g_+)$ is isometric to the hyperbolic space.
\end{proposition}

\begin{proof}
Indeed we may assume that $Q_1 =Q_{1}^{\mathbb{S}}=1$ by conformal transformation. Let $\bar{g}=\bar{\rho}^2 g_+$ and recall
\[
\bar{T}=\frac{1-|d\bar{\rho}|_{\bar{g}}^2}{\bar{\rho}}\geq 2 Q_1=2.
\] 
Then $\bar{\rho}^{-2}|d \bar{\rho}|_{+}^2=|d\bar{\rho}|_{\bar{g}}^2=1-T\bar{\rho}^2\leq 1-2\bar{\rho}$, i.e. $\bar{\rho}\leq \frac{1}{2}$. Now we can define 
$$\bar{u}:=\varphi(\bar{\rho})\bar{\rho}^{\frac{n-2\gamma}{2}},\quad \bar{\rho}\in \left[0,\frac{1}{2}\right].$$
Here $\varphi(\bar{\rho})$ is defined by equation (\ref{dphi}).
Then $\bar{u}$ has asymptotics
\[
\bar{u}=r^{\frac{n}{2}-\gamma}\left(1+\bar{u}_{2\gamma}r^{2\gamma}+\bar{u}_2 r^2 +o(r^2)\right),\quad \textrm{as } r\to 0.
\]
where
\[
\begin{aligned}
\bar{u}_{2\gamma}&=\frac{n-2\gamma}{2d_{\gamma}}Q_{2\gamma}^{\mathbb{S}},\\
\bar{u}_2 &=\frac{n-2\gamma}{4}(\hat{J} -J^{\mathbb{S}})-\frac{n-2\gamma}{4(2\gamma-2)}J^{\mathbb{S}}.
\end{aligned}
\]
Moreover it is easy to check that
\[
\Delta_+ \bar{u}+\frac{n^2 -(2\gamma)^2}{4}\bar{u}=\rho^{\frac{n-2\gamma}{2}}\left(\bar{A}_{\gamma}(\bar{\rho})\bar{\rho}^{-2}|d\bar{\rho}|_{+}^2 +\bar{B}_{\gamma}(\bar{\rho})\right),
\]
where $\bar{A}_\gamma (\bar{\rho})$ and $\bar{B}_{\gamma}(\bar{\rho})$ are defined by equation (\ref{phiab}). Then $\bar{T}\geq 2$ implies that
\[
\bar{A}_{\gamma}(\bar{\rho})\bar{\rho}^{-2}|d\bar{\rho}|_{+}^2 +\bar{B}_{\gamma}(\bar{\rho})\geq \bar{A}_{\gamma}(\bar{\rho})\cdot(1-2\bar{\rho})+\bar{B}_{\gamma}(\bar{\rho})=0
\]
since $\bar{A}_{\gamma}(\bar{\rho})\leq 0$ by Lemma \ref{bpos}. Therefore
\begin{equation}\label{b-u}
\Delta_+ \bar{u}+\frac{n^2 -(2\gamma)^2}{4}\bar{u} \geq 0.
\end{equation}
Let $\bar{w}$ satisfy the linear equation
\[
\Delta_+ \bar{w}+\frac{n^2 -(2\gamma)^2}{4}\bar{w} = 0,\quad \bar{w}\sim r^{\frac{n-2\gamma}{2}} \textrm{ as }r\to 0.
\]
Then $w$ has asymptotics as follow
\[
\bar{w}=r^{\frac{n-2\gamma}{2}}\left(1+ \bar{w}_{2\gamma}r^{2\gamma}+\bar{w}_2 r^2 +o(r^2)\right),
\]
where
\[
\begin{aligned}
\bar{w}_{2\gamma}&=\frac{n-2\gamma}{2d_{\gamma}}Q_{2\gamma},\\
\bar{w}_2 &=-\frac{n-2\gamma}{4(2\gamma-2)}\hat{J}.\\
\end{aligned}
\]
Since $Q_1=1$ and $\hat{J}\leq \frac{n}{2}$, it is easy to check that
\[
\bar{w}_2 \leq \bar{u}_2 \quad \Leftrightarrow\quad \frac{2\gamma-1}{2(1-\gamma)}\left(\hat{J}-\frac{n}{2}\right)\leq 0.
\]
Note that $\frac{\bar{w}}{\bar{u}}$ satisfies the equation:
\[
-\Delta_+ \left(\frac{\bar{w}}{\bar{u}}\right)=\left(\frac{n^2 -(2\gamma)^2}{4}+\frac{\Delta_+ \bar{u}}{\bar{u}}\right)\frac{\bar{w}}{\bar{u}}+2\left\langle d\left(\frac{\bar{w}}{\bar{u}}\right),\frac{d \bar{w}}{\bar{w}}\right\rangle_{g_+},\quad \left(\frac{\bar{w}}{\bar{u}}\right)\big{|}_{M}=1.
\]
Notice that $\frac{\bar{w}}{\bar{u}}>0$. Applying maximum principle according to equation (\ref{b-u}), we show that $\frac{\bar{w}}{\bar{u}}$ can not attain an interior positive minimum. Hence $\frac{\bar{w}}{\bar{u}}\geq 1$, i.e. $\bar{w}\geq \bar{u}$. Thus near the boundary, 
\[
1+\bar{w}_{2\gamma}r^{2\gamma} +\bar{w}_2 r^2 +o(r^2)\geq 1 +\bar{u}_{2\gamma} r^{2\gamma}+\bar{u}_2 r^2 +o(r^2).
\]
By $\bar{w}_2 \leq \bar{u}_2$, we must have $\bar{w}_{2\gamma} \geq \bar{u}_{2\gamma}$, i.e. on $M$,
\[
Q_{2\gamma}\leq Q_{2\gamma}^{\mathbb{S}}.
\]
since $d_\gamma <0$.

If $Q_{2\gamma}= Q_{2\gamma}^{\mathbb{S}}$, then we must have $\bar{w}_2 = \bar{u}_2$ which implies that 
$ \hat{J}=\frac{n}{2}= \frac{n}{2} Q_{1}^2$. Thus by Lemma \ref{btlem}, $(X^{n+1},g_+)$ is isometric to the hyperbolic space.
\end{proof}

\vspace{0.2in}
Now we are ready to prove Theorem \ref{mth1}.

\begin{proof}[Proof of Theorem \ref{mth1}]
Firstly, we assume that $Y_1 (M,[\hat{g}])$ is attained by $\hat{g}$. If $Y_1 (M,[\hat{g}])=0$, then $Q_1 =0$ for $\hat{g}$. By Lemma \ref{q1null}, $Q_{2\gamma} \leq 0$ and $Q_{2\gamma}(p)<0$ for some $p\in M$. Thus
\[
Y_{2\gamma}(M,[\hat{g}])< 0.
\]

If $Y_1 (M,[\hat{g}])>0$, by conformal transformation, we assume $Q_{1}=1$ in terms of $\hat{g}$, then
\[
Y_1 (M,[\hat{g}])=\frac{n-1}{2}\textrm{Vol}(M,\hat{g})^{\frac{1}{n}}\quad \Rightarrow \textrm{Vol}(M,\hat{g})^{\frac{1}{n}}=\frac{2}{n-1}Y_1 (M,[\hat{g}]).
\]
It follows that
\begin{equation}\label{prey1}
Y_{2\gamma}(M,[\hat{g}])\leq \frac{n-2\gamma}{2}Q_{2\gamma}^{\mathbb{S}}\textrm{Vol}(M,\hat{g})^{\frac{2\gamma}{n}}
= \left(\frac{2}{n-1}\right)^{2\gamma}\frac{\Gamma\left(\frac{n}{2}+\gamma\right)}{\Gamma\left(\frac{n}{2}-\gamma\right)}\left(Y_{1} (M,[\hat{g}])\right)^{2\gamma},
\end{equation}
where the first inequality dues to Proposition \ref{qpn1}.

If the equality of (\ref{prey1}) holds, we must have $Q_{2\gamma}= Q_{2\gamma}^{\mathbb{S}} $. Then by Proposition \ref{qpn1} again, $(X^{n+1},g_+)$ must be isometric to the hyperbolic space.
\end{proof}

\vspace{0.2in}
\section{$Y_{2}(M,[\hat{g}])$ vs. $Y_{2\gamma}(M,[\hat{g}])$}\label{y2}

In this part, for Poincar\'{e}-Einstein manifold $(X^{n+1},g_+)(n\geq 3)$ with conformal infinity $(M,[\hat{g}])$, we assume that $Y_{2}(M,[\hat{g}])\geq 0$ is attained by $\hat{g}$ and $r$ is the geodesic defining function for $\hat{g}$. Take $Q_{2\gamma} = Q_{2\gamma}^{\hat{g}}$, $\hat{J}=J^{\hat{g}}$ and $\hat{A}=A^{\hat{g}}$ for short.

\begin{itemize}
\item Let $\tilde{v}$ satisfy the linear equation
\[
\Delta_+ \tilde{v}-(n+1)\tilde{v}=0, \quad \tilde{v}\sim r^{-1} \textrm{ as } r\to 0.
\] 
Then 
\[
\tilde{v}=r^{-1}\left(1+\tilde{v}_2 r^2 +\tilde{v}_4 r^4 +o(r^4)\right),\quad \textrm{as } r\to 0.
\]
where
\[
\begin{aligned}
&\tilde{v}_2 =\frac{1}{2n}\hat{J},\\
&\tilde{v}_{4}=\frac{1}{8n(n-2)}\left(-\hat{J}^2 +n|\hat{A}|^{2}\right).
\end{aligned}
\]
\item Let $\tilde{\rho}:=\tilde{v}^{-1}$, then
\[
\tilde{\rho}=r\left(1+\tilde{\rho}_2 r^2 +\tilde{\rho}_4 r^4 +o(r^4)\right),\quad \textrm{as } r\to 0.
\]
where
\[
\begin{aligned}
&\tilde{\rho}_2 =-\frac{1}{2n}\hat{J},\\
&\tilde{\rho}_{4}=-\frac{1}{8n(n-2)}\left(-\hat{J}^2 +n|\hat{A}|^{2}\right)+\frac{1}{4n^2}\hat{J}^2.
\end{aligned}
\]
\item Let $\tilde{g}=\tilde{\rho}^2 g_+$, and denote
\[
\tilde{T}=\frac{1-|d\tilde{\rho}|_{\tilde{g}}^2}{\tilde{\rho}^{2}},
\]
then it is easy to compute that $\tilde{T}$ satisfies 
\[
\Delta_{\tilde{g}}\tilde{\rho}=-(n+1)\tilde{T}
\]
and as well as in \cite{Le,CC},
\begin{equation}\label{tee}
\Delta_+ \tilde{T}=-\frac{2}{(n-1)^2}\tilde{\rho}^2 |\tilde{E}|_{\tilde{g}}^2, \quad \tilde{T}|_M =\frac{2}{n}\hat{J},    
\end{equation}
where $\tilde{E}=E_{\tilde{g}}$ for short.
\end{itemize}

\begin{lemma}\label{jnull}
Let $\delta \in (0,1)$ and $\gamma\in (1,\min\{n/2,2\})$, if $\hat{J}=0$, then
\[
Q_{2\gamma} \leq 0 \leq Q_{2\delta}.
\]
Moreover, there exists a point $p\in M$ such that $Q_{2\gamma}<0$. 
\end{lemma}

\begin{proof}
The case $\gamma\in (1,\min\{n/2,2\})$ has been proved in \cite[Corollary 6.5, Corollary 6.6]{CC}.
Here we provide another proof of $Q_{2\gamma} \leq 0$ for case $\gamma\in (1,\min\{n/2,2\})$, In addition, we omit the proof of $ Q_{2\delta}\geq 0$ which is similar to our proof method. 
Let $\tilde{w}$ satisfy
\[
\Delta_+ \tilde{w}+\frac{n^2 -(2\gamma)^2}{4}\tilde{w} = 0,\quad \tilde{w}\sim r^{\frac{n-2\gamma}{2}} \textrm{ as }r\to 0.
\]
Then $w$ has the asymptotics 
\[
\tilde{w}=r^{\frac{n-2\gamma}{2}}\left(1 +\tilde{w}_{2\gamma}r^{2\gamma}+ \tilde{w}_4 r^4 +o(r^4)\right),
\]
where
\[
\begin{aligned}
\tilde{w}_{2\gamma}&=\frac{n-2\gamma}{2d_{\gamma}}Q_{2\gamma},\\
\tilde{w}_4 &= \frac{n-2\gamma}{32(2-\gamma)(n-2)^2}|E^{\hat{g}}|_{\hat{g}}^2,\\
\end{aligned}
\]
since $\hat{J} =0$.
Recall that $\tilde{v}$ satisfies 
\[
\Delta_+ \tilde{v}-(n+1)\tilde{v}=0,
\]
and 
\[
\tilde{v}=r^{\frac{n-1}{2}}\left(1 + \frac{1}{8(n-2)^3}|E^{\hat{g}}|_{\hat{g}}^2 r^4 +o(r^4)\right).
\]
Then $\tilde{g}=\tilde{\rho}^{2}g_+$ with $\tilde{\rho}=\tilde{v}^{-1}$. According to Lemma \ref{btlem}, we have
\[1-\tilde{v}^{-2}|d \tilde{v}|_{+}^2=1-|d\tilde{\rho}|_{\tilde{g}}^2 \geq 0.\]
Define
\[
\tilde{\xi}:= \tilde{v}^{-\frac{n-2\gamma}{2}}.
\]
Then
\[
\tilde{\xi}=r^{\frac{n-2\gamma}{2}}\left(1-\frac{n-2\gamma}{16(n-2)^3}|E^{\hat{g}}|_{\hat{g}}^2 r^4 +o(r^4)\right).
\]
By direct computation, we have
\[
\Delta_+ \tilde{\xi}+\frac{n^2 -(2\gamma)^2}{4}\tilde{\xi} =-\frac{(n-2\gamma)(n+2-2\gamma)}{4}\tilde{\xi} \left(1-\tilde{v}^{-2}|d \tilde{v}|_{+}^2 \right)\leq 0.
\]
Note that $\frac{\tilde{w}}{\tilde{\xi}}$ satisfies the equation:
\[
-\Delta_+ \left(\frac{\tilde{w}}{\tilde{\xi}}\right)=\left(\frac{n^2 -(2\gamma)^2}{4}+\frac{\Delta_+ \tilde{\xi}}{\tilde{\xi}}\right)\frac{\tilde{w}}{\tilde{\xi}}+2\left\langle d\left(\frac{\tilde{w}}{\tilde{\xi}}\right),\frac{d \tilde{w}}{\tilde{w}}\right\rangle_{g_+},\quad \left(\frac{\tilde{w}}{\tilde{\xi}}\right)\big{|}_{M}=1.
\]
Notice that $\frac{\tilde{w}}{\tilde{\xi}}>0$. Applying maximum principle, $\frac{\tilde{w}}{\tilde{\xi}}\leq 1$, i.e. $\tilde{w}\leq \tilde{\xi}$. Thus near the boundary, 
\[
1+\tilde{w}_{2\gamma}r^{2\gamma} +\tilde{w}_4 r^4 +o(r^4)\leq 1 -\frac{n-2\gamma}{16(n-2)^3}|E^{\hat{g}}|_{\hat{g}}^2 r^4 +o(r^4).
\]
Therefore $\tilde{w}_{2\gamma}\leq 0$ which implies that $Q_{2\gamma}\leq 0$ since $d_\gamma>0$.
\end{proof}

\begin{proposition}\label{qpn2}
Let $\delta \in (0,1)$ and $\gamma\in (1,\min\{n/2,2\})$, if $\hat{J}$ is a positive constant, then
\[
\frac{Q_{2\gamma}}{Q_{2\gamma}^{\mathbb{S}}}\leq \frac{\hat{J}}{J^{\mathbb{S}}}\leq \frac{Q_{2\delta}}{Q_{2\delta}^{\mathbb{S}}}.
\]
\end{proposition}
\begin{proof}
We only deal with the case  of $\gamma\in (1,\min\{n/2,2\})$ and the proof of the other case of $\delta \in (0,1)$ is similar. By conformal transformation, We may assume that $\hat{J}=J^{\mathbb{S}}=\frac{n}{2}$. Let $\tilde{g}=\tilde{\rho}^2 g_+$, and recall
\[
\tilde{T}=\frac{1-|d \tilde{\rho}|_{\tilde{g}}^2}{\tilde{\rho}^2},\quad \tilde{T}|_{M}=\frac{2}{n}\hat{J}=1.
\] 
By maximum principle we have $\tilde{T}\geq 1$, then $\tilde{\rho}^{-2}|d \tilde{\rho}|_{+}^2=|d\tilde{\rho}|_{\tilde{g}}^2=1-\tilde{T}\tilde{\rho}^2\leq 1-\tilde{\rho}^2$, i.e. $\tilde{\rho}\leq 1$. Then we can define
$$\tilde{u}:=\psi(\tilde{\rho})\cdot \tilde{\rho}^{\frac{n-2\gamma}{2}},\quad \tilde{\rho}\in [0,1].$$
Here the definition of $\psi(\tilde{\rho})$ comes from equation (\ref{dpsi}).
Similarly $\tilde{u}$ has asymptotics
\[
\tilde{u}=r^{\frac{n}{2}-\gamma}\left(1+\tilde{u}_2 r^2 +\tilde{u}_{2\gamma} r^{2\gamma}+\tilde{u}_4 r^4 +o(r^4)\right),\quad \textrm{as } r\to 0.
\]
where
\[
\begin{aligned}
&\tilde{u}_2 =-\frac{n-2\gamma}{4(2\gamma-2)}\hat{J},\quad \tilde{u}_{2\gamma}=\frac{n-2\gamma}{2d_{\gamma}}Q_{2\gamma}^{g_{\mathbb{S}}},\\
&\tilde{u}_{4}=-\frac{n-2\gamma}{16(n-2)^3}|E^{\hat{g}}|_{\hat{g}}^2 +\frac{n-2\gamma}{16(2\gamma-2)(2\gamma-4)}\left(\frac{n-2\gamma+4}{2}-\frac{2\gamma-2}{n}\right)(J^{\mathbb{S}})^2.
\end{aligned}
\]
And it is easy to check that
\[
\Delta_+ \tilde{u}+\frac{n^2 -(2\gamma)^2}{4}\tilde{u}=\tilde{\rho}^{\frac{n-2\gamma}{2}}\left(\tilde{A}_{\gamma}(\tilde{\rho})\tilde{\rho}^{-2}|d\tilde{\rho}|_{+}^2 +\tilde{B}_{\gamma}(\tilde{\rho})\right)
\]
where $\tilde{A}_\gamma$ and $\tilde{B}_{\gamma}$ are defined by equation (\ref{psi-ef}). Then $\tilde{T}\geq 1$ implies that
\[
\tilde{A}_{\gamma}(\tilde{\rho})\tilde{\rho}^{-2}|d\tilde{\rho}|_{+}^2 +\tilde{B}_{\gamma}(\tilde{\rho})\leq \tilde{A}_{\gamma}(1-\tilde{\rho}^2)+\tilde{B}_{\gamma}(\tilde{\rho})=0
\]
since $\tilde{A}_{\gamma}(\tilde{\rho})\geq 0$ by Lemma \ref{fneg}. Therefore
\begin{equation}\label{ttu}
\Delta_+ \tilde{u}+\frac{n^2 -(2\gamma)^2}{4}\tilde{u} \leq 0.
\end{equation}
Let $\tilde{w}$ satisfy the linear equation
\begin{equation}\label{twe}
\Delta_+ \tilde{w}+\frac{n^2 -(2\gamma)^2}{4}\tilde{w} = 0,\quad \tilde{w}\sim r^{\frac{n-2\gamma}{2}} \textrm{ as }r\to 0.
\end{equation}
Then $\tilde{w}$ has asymptotics as follow
\[
\tilde{w}=r^{\frac{n-2\gamma}{2}}\left(1+\tilde{w}_2 r^2 +\tilde{w}_{2\gamma}r^{2\gamma}+\tilde{w}_4 r^4 +o(r^4)\right),
\]
where
\[
\begin{aligned}
&\tilde{w}_2 =-\frac{n-2\gamma}{4(2\gamma-2)}\hat{J},\quad \tilde{w}_{2\gamma}=\frac{n-2\gamma}{2d_{\gamma}}Q_{2\gamma},\\
&\tilde{w}_{4}=-\frac{n-2\gamma}{16(2\gamma-4)(n-2)^2}|E^{\hat{g}}|^{2}_{\hat{g}} +\frac{n-2\gamma}{16(2\gamma-2)(2\gamma-4)}\left(\frac{n-2\gamma+4}{2} -\frac{2\gamma-2}{n}\right)\hat{J}^2.
\end{aligned}
\]
Since $\hat{J}=J^{\mathbb{S}}=\frac{n}{2}$, it is easy to check that
\[
\tilde{w}_2 =\tilde{u}_2,\quad\textrm{and }\quad\tilde{w}_4 \geq \tilde{u}_4 \left(\Leftrightarrow \frac{1}{4-2\gamma}|E^{\hat{g}}|_{\hat{g}}^2\geq -\frac{1}{n-2}|E^{\hat{g}}|_{\hat{g}}^2\right).
\]
Note that $\frac{\tilde{w}}{\tilde{u}}$ satisfies the equation:
\begin{equation}\label{uw2}
-\Delta_+ \left(\frac{\tilde{w}}{\tilde{u}}\right)=\left(\frac{n^2 -(2\gamma)^2}{4}+\frac{\Delta_+ \tilde{u}}{\tilde{u}}\right)\frac{\tilde{w}}{\tilde{u}}+2\left\langle\nabla\left(\frac{\tilde{w}}{\tilde{u}}\right),\frac{\nabla \tilde{w}}{\tilde{w}}\right\rangle_{g_+},\quad \left(\frac{\tilde{w}}{\tilde{u}}\right)\big{|}_{M}=1.
\end{equation}
Notice that $\frac{\tilde{w}}{\tilde{u}}>0$. Applying maximum principle according to equation (\ref{ttu}), we show that $\frac{\tilde{w}}{\tilde{u}}$ can not attain an interior positive maximum. Hence $\frac{\tilde{w}}{\tilde{u}}\leq 1$, i.e. $\tilde{w}\leq \tilde{u}$. Thus near the boundary, 
\[
1+\tilde{w}_2 r^2 +\tilde{w}_{2\gamma}r^{2\gamma} +\tilde{w}_4 r^{4}+o(r^4)\leq 1+\tilde{u}_2 r^2 +\tilde{u}_{2\gamma} r^{2\gamma}+\tilde{u}_4 r^4 +o(r^4).
\]
According to $\tilde{w}_2 =\tilde{u}_2$ and $\tilde{w}_4 \geq \tilde{u}_4$, we have $\tilde{w}_{2\gamma} \leq \tilde{u}_{2\gamma}$, i.e. 
\[
Q_{2\gamma}\leq Q_{2\gamma}^{\mathbb{S}}, \quad \textrm{on $M$,}
\]
since $d_{\gamma}>0$.
\end{proof}

\begin{lemma}\label{qqs}
If $\hat{J}=J^{\mathbb{S}}=\frac{n}{2}$ and $Q_{2\gamma}=Q_{2\gamma}^{\mathbb{S}}$ with $\gamma\in (1,\min\{n/2,2\})$, then $(X, g_+)$ is isometric to the hyperbolic space $\mathbb{H}^{n+1}$.
\end{lemma}
\begin{proof}
Let $\tilde{r}=\tilde{u}^{\frac{2}{n-2\gamma}}$ and $g=\tilde{r}^2 g_+$, recall equation (\ref{ttu}), we have
\begin{equation}\label{tre}
\Delta_{g} \tilde{r}+\frac{n+2\gamma}{2}\frac{1-|d\tilde{r}|^{2}_{g}}{\tilde{r}}\leq 0.
\end{equation}

By \cite[Section 3]{CC}, consider the metric measure space $(X^{n+1},g,\tilde{r}^{m_{0}}dv_{g},m_{0}-1)$ with $m_0=1-2\gamma$. Due to  \cite[Lemma 3.2]{CC}, 
\[
J_{g}=-\frac{\Delta_g \tilde{r}}{\tilde{r}}-\frac{n+1}{2}\left(\frac{1-|d \tilde{r}|^{2}_{g}}{\tilde{r}^2}\right),
\]
and 
\[
J_{\phi_0}^{m_0}=J_g-\frac{m_0}{n+1}\left(J_g+\tilde{r}^{-1}\Delta_g\tilde{r}\right).
\]
where $J_{\phi_0}^{m_0}$ is defined in \cite[Section 3]{CC}.
It follows by equation (\ref{tre}) that
\[
J_{\phi_0}^{m_0}\geq 0.
\]

Recall equation (\ref{uw2}), denote $\tilde{U}=\frac{\tilde{w}}{\tilde{u}}$, i.e. $\tilde{w}=\tilde{r}^{\frac{n-2\gamma}{2}}\tilde{U}$, then $\tilde{U}$ satisfies
\[
L_{2,\phi_0}^{m_0}\tilde{U}=\tilde{r}^{-\frac{m_0 +n+3}{2}}\left(-\Delta_+ -\frac{n^2 -(2\gamma)^2}{4}\right)\left(\tilde{r}^{\frac{m_0 +n-1}{2}}\tilde{U}\right)=0.
\]
where $L_{2,\phi_0}^{m_0}$ is defined in \cite[equation (3.2) in Section 3]{CC}.
Moreover $\tilde{U}$ has asymptotics
\[
\tilde{U}=1+\frac{n-2\gamma}{2d_{\gamma}}\left(Q_{2\gamma}-Q_{2\gamma}^{\mathbb{S}}\right)r^{2\gamma}+O(r^4),\quad \textrm{as }r\to 0.
\]
Then we compute that
\[
\int_{X}\left(|\nabla \tilde{U}|^{2}_{g}+\frac{m_0 +n-1}{2}J_{\phi_0}^{m_0}\tilde{U}^2 \right)\tilde{r}^{m_0}dv_{g}=\int_{X}\tilde{U}\left(L_{2,\phi_0}^{m_0}\tilde{U}\right) \tilde{r}^{m_0}dv_g +\int_{M}\tilde{r}^{m_0} \tilde{U}\frac{\partial \tilde{U}}{\partial\nu}dv_{\hat{g}},
\]
where $\nu$ is the outer unit normal vector for $g$.
Since $L_{2,\phi_0}^{m_0}\tilde{U}=0$, $\tilde{U}|_{M}=1$ and $\frac{\partial}{\partial\nu}=-\frac{\partial}{\partial r}$, we have
\[
0\leq \int_{X}\left(|\nabla \tilde{U}|^{2}_{g}+\frac{m_0 +n-1}{2}J_{\phi_0}^{m_0}\tilde{U}^2 \right)\tilde{r}^{m_0}dv_{g}=-\frac{n-2\gamma}{2d_{\gamma}} \int_{M} \left(Q_{2\gamma}-Q_{2\gamma}^{\mathbb{S}}\right)dv_{\hat{g}}=0.
\]
Thus $\tilde{U}=1$ and $J_{\phi_0}^{m_0}=0$. Then the equality of equation (\ref{ttu}) holds,  we have $\tilde{T}=1$ and furthermore $\tilde{E}_{\tilde{g}}=0$ by equation (\ref{tee}). Recall
\[
\tilde{E}_{\tilde{g}}=-(n-1)\tilde{\rho}\left(\tilde{\nabla}^{2}\tilde{\rho}-\frac{\Delta_{\tilde{g}}\tilde{\rho}}{n+1}\right), \quad \Delta_{\tilde{g}}\tilde{\rho}=-(n+1)\tilde{T}
\]
which imply that
\[
\tilde{\nabla}^2 \tilde{\rho}+\tilde{\rho}\tilde{g}=0,\quad \tilde{\rho}|_{M}=0.
\]
By the standard argument of Obata equations with boundary or \cite{CLW2}, we have $(X,\tilde{g})$ is isometric to the upper half sphere $(\mathbb{S}_{+}^{n+1},g_{\mathbb{S}})$, i.e. $(X, g_+)$ is isometric to the hyperbolic space $\mathbb{H}^{n+1}$.
\end{proof}

\vspace{0.2in}
Now we are ready to prove Theorem \ref{mth2}.

\begin{proof}[Proof of Theorem \ref{mth2}]
Firstly if $Y_{2}(M,[\hat{g}])=0$, then $\hat{J}=0$, by Lemma \ref{jnull}, we have $Q_{2\gamma}\leq 0$ and $Q_{2\gamma}(p)< 0$ for some $p\in M$. Thus
\[
Y_{2\gamma}(M,[\hat{g}])< 0.
\]
If $Y_{2}(M,[\hat{g}])>0$, by conformal transformation, assume that the Yamabe metric $\hat{g}$ satisfies $\hat{J}=\frac{n}{2}$, then
\[
Y_2 (M,[\hat{g}])=\frac{n(n-2)}{4}\textrm{Vol}(M,\hat{g})^{\frac{2}{n}}\quad \Rightarrow \textrm{Vol}(M,\hat{g})^{\frac{2}{n}}=\frac{4}{n(n-2)}Y_2 (M,[\hat{g}]).
\]
Recall the definition of $Y_{2\gamma}(M,[\hat{g}])$, it follows that
\[
Y_{2\gamma}(M,[\hat{g}])\leq \frac{n-2\gamma}{2}Q_{2\gamma}^{\mathbb{S}}\textrm{Vol}(M,\hat{g})^{\frac{2\gamma}{n}}
= \left(\frac{4}{n(n-2)}\right)^{\gamma}\frac{\Gamma\left(\frac{n}{2}+\gamma\right)}{\Gamma\left(\frac{n}{2}-\gamma\right)}\left(Y_{2} (M,[\hat{g}])\right)^{\gamma},\]
where the first inequality dues to Proposition \ref{qpn2}.

If the equality holds, then $Q_{2\gamma}=Q_{2\gamma}^{\mathbb{S}}$. By Lemma \ref{qqs}, we have $(X, g_+)$ is isometric to the hyperbolic space $\mathbb{H}^{n+1}$.

\end{proof}



\begin{thebibliography}{99}


%
%

\bibitem[Be]{Be1} 
Beckner W.
{Sharp Sobolev inequalities on the sphere and the Moser-Trudinger inequality}.
Ann of Math, 1993, {138}(2): 213-243.

\bibitem[CC]{CC}
Case J, Chang S-Y A.
{On the fractional GJMS operators},
Commun Pure Appl Math, 2016, {69}(6): 1017-1061.

\bibitem[CDLS]{CDLS} 
Chru\'{s}ciel P, Delay E,  Lee J,  Skinner D.
{Boundary regularity of conformal compact Einstein metrics}.
J Diff Geom, 2005, {69}: 111-136.

%

\bibitem[CLW1]{CLW}
 Chen X, Lai M, Wang F. 
{Escobar-Yamabe compactification for Poincare-Einstein manifolds and rigidity theorems}. 
Adv Math, 2019, {343}: 16-35.

\bibitem[CLW2]{CLW2}
 Chen X, Lai M, Wang F. 
{The Obata equation with Robin boundary condition}. 
Revista Matem{\'a}tica Iberoamericana, 2021, {2}: 643–670.



\bibitem[CS]{CS}
Caffarelli L. Silvestre L.. 
{An extension problem related to the fractional Laplacian}.
Comm. Partial Differential Equations, 2007, {32}(7-9):1245–1260.



\bibitem[DJ]{DJ}
Dutta S, Javaheri M. 
{Rigidity of conformally compact manifolds with the round sphere as conformal infinity}. 
Adv Math, 2010, {224}: 525-538.

\bibitem[Es1]{Es1}
Escobar J. 
{The Yamabe problem on manifolds with boundary}. 
J Differ Geom, 1992, {35}(1): 21-84.

\bibitem[Es2]{Es2}
Escobar J. 
{Conformal deformation of a Riemannian metric to a scalar flat metric with constant mean curvature on the boundary}. 
Ann of Math (2), 1992, {136}(1): 1-50. (With an addendum, Ann of Math (2), 1994, {139}(3): 749-750.)


\bibitem[Go]{Go}
Gonz\'{a}lez, M M. {Recent Progress on the Fractional Laplacian in Conformal Geometry}. Recent Developments in Nonlocal Theory, edited by Giampiero Palatucci and Tuomo Kuusi, Warsaw, Poland: De Gruyter Open Poland, 2017, pp. 236-273. 


\bibitem[GH]{GH}
Gursky M J, Han Q. 
{Non-existence of Poincar\'{e}-Einstein manifolds with prescribed conformal infinity}. 
Geom Funct Anal, 2017, {27}(4): 863-879.

\bibitem[GJMS]{GJMS}
Graham C R, Jenne R, Mason L J,  Sparling G A J.
{Conformally invariant Powers of the Laplacian. I. Existence}. 
J London Math Soc (2),  1992, {46}(3): 557-565. 


\bibitem[GQ]{GQ} 
Guillarmou C, Qing J. 
{Spectral Characterization of Poincar\'{e}-Einstein manifolds with infinity of positive Yamabe type}.  
Int Math Res Notices, 2010,  9: 1720-1740.



\bibitem[GZ]{GZ} 
Graham C R,  Zoworski M. 
{Scattering matrix in conformal geometry}. 
 Invent Math, 2003,  {152}:  89-118.


%
%

\bibitem[JS]{JS}
Joshi M, S\'{a} Barreto A. 
{Inverse scattering on asymptotically hyperbolic manifolds}.  
Acta Math, 2000,  {184}: 41-86.

 

\bibitem[Le]{Le} 
Lee J. 
{The spectrum of an asymptotically hyperbolic Einstein manifold}. 
Comm  Anal Geom, 1995,  {3}(1-2): 253-271.

%
\bibitem[LQS]{LQS} 
 Li G, Qing J, Shi Y. 
{Gap phenomena and curvature estimates for conformally compact Einstein manifolds}. 
Trans Amer Math Soc, 2017,  {369}(6): 4385-4413.
%
%
%

\bibitem[MM]{MM} 
Mazzeo R,  Melrose R. 
{Meromorphic extension of the resolvent on complete spaces with asymptotically constant negative curvature}. 
J Funct Anal, 1987,   {75}: 260-310.


%
%
%

%
%
%
%



\bibitem[WW]{WW}
Wang X, Wang Z.
{On a sharp inequality relating Yamabe invariants on a Poincare-Einstein manifold}.
Proc. Amer. Math. Soc. 150 (2022), 4923-4929.



\bibitem[WZ1]{WZ1}
Wang F , Zhou H. 
{Comparison theorems for GJMS operators}. 
Sci. China Math. 64, 2479–2494 (2021).

\bibitem[WZ2]{WZ2}
Wang F , Zhou H. 
{Lower bound for the relative volume of Poincar\'{e}-Einstein manifolds}.
arXiv:2112.06669.

\end{thebibliography}
\end{document}